\documentclass[reqno, 11pt]{amsart}
\usepackage[utf8]{inputenc}
\usepackage{amssymb,amsmath,amsfonts,amsthm,calrsfs,mathtools}
\usepackage[all]{xy}
\usepackage{color}
\usepackage[english]{babel}
\usepackage[T1]{fontenc}
\usepackage{multicol}
\usepackage{latexsym}
\usepackage{enumitem}
\usepackage{tabularx}
\usepackage{mdframed}
\usepackage[a4paper,top=3cm,bottom=2cm,left=3cm,right=3cm,marginparwidth=1.75cm]{geometry}
\usepackage[colorinlistoftodos]{todonotes}
\usepackage[colorlinks=true, allcolors=red]{hyperref}
\theoremstyle{plain}
\newtheorem{theorem}{Theorem}
\newtheorem{lemma}{Lemma}

\newtheorem{proposition}{Proposition}

\newtheorem{corollary}{Corollary}

\theoremstyle{definition}

\newtheorem{definition}{Definition}

\newtheorem{remark}{Remark}

\providecommand{\keywords}[1]

\title{Li--Yorke chaotic weighted composition operators on Hardy and weighted Bergman spaces over the unit disk}
\author{Carlos F. \'{A}lvarez}

\address{Departamento de Matem\'{a}ticas, Universidad del Atl\'{a}ntico, 
Puerto Colombia, Colombia}

\email[{(Carlos F. \'Alvarez)}]{cfalvarez@mail.uniatlantico.edu.co}
\email[{(Juan Manzur)}]{jcmanzur@mail.uniatlantico.edu.co}

\author{Jo\~{a}o R. Carmo}

\author{Juan Manzur} 

\address{Instituto Federal de Educação, Ciência e Tecnologia da Bahia, Seabra-Bahía, Brasil }
\email{joaoribeiro@ifba.edu.br}

\makeatletter

\date{\today}

\begin{document}

\begin{abstract}
We study Li--Yorke and mean Li--Yorke chaos for weighted composition operators $C_{w,\varphi}$ on Banach spaces of analytic functions on the unit disk $\mathbb{D}$. Under natural conditions on the space, we show that $C_{w,\varphi}$ is (densely) Li--Yorke chaotic if and only if it is not power-bounded, and (densely) mean Li--Yorke chaotic if and only if it is not absolutely Ces\`aro bounded. These results are applied to Hardy spaces $H^p(\mathbb{D})$, $1 \le p \le \infty$, and weighted Bergman spaces $A^p_\beta(\mathbb{D})$, $-1 < \beta < \infty$ and $1 < p < \infty$.
\end{abstract}

\maketitle
\markright{LI--YORKE CHAOTIC WEIGHTED COMPOSITION OPERATORS}


\section{Introduction}

The study of chaos for infinite-dimensional dynamical systems has grown into an active research area at the intersection of functional analysis, operator theory and topological dynamics; see \cite{BayartMatheron, GrossePeris} for a thorough account of the subject.

The notion of chaos motivating this work was introduced by Li and Yorke \cite{LiYorke1975}, who showed that a period-three orbit forces the coexistence of orbits of every period together with an uncountable set of trajectories that neither converge together nor diverge
to infinity. This phenomenon, known as \emph{Li--Yorke chaos}, captures the coexistence of pairs of orbits that are recurrently close yet persistently apart; it has since been extended to linear operators on infinite-dimensional Banach and Fr\'echet spaces
\cite{Bermudez2011, Bernardes2015}. Beyond this notion, several related forms of chaotic behavior have been studied, including distributional chaos, generic chaos, and mean Li--Yorke chaos \cite{Bermudez2011, BernardesBonillaMullerPeris2013, Bernardes2015,
BernardesPerisSanchez, BernardesBonillaPinto, BernardesVasconcellos2025}.

Weighted composition operators constitute a central and well-studied class in operator theory. Given an analytic self-map $\varphi \colon \mathbb{D} \to \mathbb{D}$ and an
analytic weight function $w \colon \mathbb{D} \to \mathbb{C}$, the weighted composition operator $C_{w,\varphi}$ acts on a space of analytic functions by
\[
    C_{w,\varphi} f(z) = w(z)\cdot f(\varphi(z)).
\]
When $w \equiv 1$, one recovers the classical composition operator $C_\varphi$, which has been extensively studied on spaces of analytic functions; see \cite{CowenMacCluer1995, Shapiro1993}. The present article is devoted to a systematic study of Li--Yorke chaos and mean Li--Yorke chaos for weighted composition operators acting on two families of classical function spaces over the unit disk $\mathbb{D}$: the
Hardy spaces $H^p(\mathbb{D})$ for $1 \le p \le \infty$, and the weighted Bergman spaces $A^p_\beta(\mathbb{D})$ for $-1 < \beta < \infty$, $1<p<\infty$. These spaces provide the natural setting for operator theory on the disk; we refer to \cite{Duren1970, Rosemblum1994} for $H^p$ and to \cite{Gori2006, HKZ2000} for weighted
Bergman spaces.

The recent works \cite{BernardesBonillaPinto, BernardesVasconcellos2025} established complete characterizations of Li--Yorke chaos, mean Li--Yorke chaos, and related notions for weighted composition operators on $L^p(\mu)$ and $C_0(\Omega)$ spaces, with
applications to weighted shifts on $c_0(\mathbb{N})$, $c_0(\mathbb{Z})$, $\ell^p(\mathbb{N})$, $\ell^p(\mathbb{Z})$, and to weighted translation operators on $C_0[1,\infty)$, $C_0(\mathbb{R})$, $L^p[1,\infty)$, and $L^p(\mathbb{R})$. The
techniques developed there rely essentially on the lattice structure of $L^p(\mu)$, the locally compact topology underlying $C_0(\Omega)$, and the combinatorial structure of weighted shifts and translations, none of which are available in spaces of analytic functions.

The present work addresses this complementary and structurally distinct setting. The key observation is that in spaces of analytic functions such as $H^p(\mathbb{D})$ and $A^p_\beta(\mathbb{D})$, the subspace $H^\infty$ plays a dual role that has no analogue
in the settings of \cite{BernardesBonillaPinto, BernardesVasconcellos2025}: it is simultaneously dense in the ambient space and provides sharp norm estimates for the iterates of $C_{w,\varphi}$ in terms of the weight iterates $w^{(n)}$. This structure is captured by two conditions (\ref{C1}) and (\ref{C2}) introduced in Section~\ref{Sec3}, which serve as the foundation of our approach. Under these conditions, we prove that
$C_{w,\varphi}$ is (densely) Li--Yorke chaotic if and only if it is not power-bounded, and (densely) mean Li--Yorke chaotic if and only if it is not absolutely Ces\`aro bounded. These abstract results are then applied to obtain sufficient conditions for Li--Yorke chaos
and mean Li--Yorke chaos for weighted composition operators on $H^p(\mathbb{D})$, $1 \leq p < \infty$, and on $A^p_\beta(\mathbb{D})$, $-1 < \beta < \infty$,
$1<p<\infty$, expressed entirely in terms of the symbols $w$ and $\varphi$. The case $p = \infty$ is treated separately in Section~\ref{Sec4}, where the lack of separability of $H^\infty$ requires an adapted argument. To illustrate the sharpness of our conditions, Section~\ref{Sec5} presents examples of weighted composition operators that are Li--Yorke chaotic on $A^2_\beta(\mathbb{D})$ and $H^p(\mathbb{D})$ for $1\leq p \leq 2$. One of these examples additionally shows that the hypotheses of Theorem~\ref{sufLY} are not
necessary for Li--Yorke chaos to occur.

\section{Li--Yorke Chaotic, Power Bounded, and Absolutely Cesàro Bounded Operators}\label{Sec2}

Let \( X \) be a Banach space and let $\mathcal{L}(X)$ be the set of all bounded linear operators on \( X \). In this section, we introduce the essential results about the different kinds of Li--Yorke chaos, as well as the notions of power-bounded and absolutely Cesàro bounded operators. Before turning to the main results of this section, we briefly recall the definition of residual set and we introduce an important result for operators defined on separable Banach spaces. 

\begin{definition}
Let \( A \) be a subset of $X$.
\begin{enumerate}
   \item[(i)]  \( A \) is said to be \textit{nowhere dense} if the interior of its closure is empty. 
    
 \item[(ii)]  \( A \) is said to be of \textit{first category} (or \textit{meager}) in \( X \) if it can be expressed as a countable union of nowhere dense sets in \( X \).
    
    
\item[(iii)] \( A \) is called \textit{residual} in \( X \) if its complement \( X \setminus A \) is of first category; that is, a residual set is the complement of a meager set.
\end{enumerate}
$X$ is a Baire space if every first category subset of \( X \) has empty interior.
\end{definition}

\begin{remark}
By definition, the intersection of residual sets is again residual. Moreover, in a Baire space, every residual set is dense.
\end{remark}

\begin{proposition}[See {\cite[Corollary~4]{Bernardes2015}}]\label{resconv}
Let \( T \in \mathcal{L}(X) \), where \( X \) is an infinte-dimensional separable Banach space. If the set of all points \( x \in X \) such that the sequence \( (T^n x) \) has a subsequence converging to zero is dense in \( X \), then it is residual in \( X \).
\end{proposition}

\subsection{Notions of Li--Yorke chaos}
An operator $T\in\mathcal{L}(X)$ is \emph{Li--Yorke chaotic} if it admits an uncountable \textit{scrambled set} $S$, meaning that for all $x, y \in S$ with $x\neq y$:
$$\liminf_{n\to \infty}\|T^{n}x-T^{n}y\|=0 \quad \text{and} \quad \limsup_{n\to \infty}\|T^{n}x-T^{n}y\|>0.$$
If the set \( S \) can be chosen to be dense in \( X \), then \( T \) is called \emph{densely Li--Yorke chaotic}.

Recall that a vector \( x \in X \) is said to be \emph{irregular} for \( T \) if
\[\liminf_{n \to \infty} \|T^n x\| = 0 \quad \text{and} \quad \limsup_{n \to \infty} \|T^n x\| = \infty.\]
By \cite[Theorem~9]{Bernardes2015}, an operator is Li--Yorke chaotic if and only if it admits an irregular vector.

\begin{definition}
An operator $T\in \mathcal{L}(X)$ satisfies the \textit{Li-Yorke Chaos Criterion} (LYCC) if there exist an increasing sequence of integers \( (n_k)_k \) and a subset \( X_0 \subset X \) such that:
\begin{itemize}
    \item[(a)] \( \displaystyle\lim_{k \to \infty} T^{n_k} x = 0 \) for all \( x \in X_0 \),
    \item[(b)] \( \sup_n \| T^n|_Y \| = \infty \), where \( Y := \operatorname{span}(X_0) \) and \( T^n|_Y \) denotes the restriction of \( T^n \) to \( Y \).
\end{itemize}
\end{definition}

One might expect the LYCC to be merely a sufficient condition for Li--Yorke chaos; however, it turns out that the two notions are actually equivalent.

\begin{theorem}[See {\cite[Theorem~8]{Bermudez2011}}]\label{LYCC}
Let \( T\in \mathcal{L}(X)\). The following assertions are equivalent:
\begin{itemize}
    \item [(i)]\( T \) is Li--Yorke chaotic.
    \item [(ii)] \( T \) satisfies the Li--Yorke Chaos Criterion.
\end{itemize}
\end{theorem}

\begin{definition}
An operator \( T\in \mathcal{L}(X) \) is said to be \emph{mean Li-Yorke chaotic} if there is an uncountable subset \( S \subset X \) (a \emph{mean Li-Yorke set} for \( T \)) such that every pair \( (x, y) \) of distinct points in \( S \) is a \emph{mean Li-Yorke pair} for \( T \), in the sense that
\[
\liminf_{N \to \infty} \frac{1}{N} \sum_{j=1}^N \| T^j x - T^j y \| = 0
\quad \text{and} \quad
\limsup_{N \to \infty} \frac{1}{N} \sum_{j=1}^N \| T^j x - T^j y \| > 0.
\]

If \( S \) can be chosen to be dense (respectively, residual) in \( X \), then we say that \( T \) is \emph{densely} (respectively, \emph{generically}) mean Li-Yorke chaotic.
\end{definition}

\begin{definition}
Given an operator \( T\in \mathcal{L}(X) \) and a vector \( x \in X \), we say that \( x \) is an \emph{absolutely mean irregular} (respectively, \emph{absolutely mean semi-irregular}) vector for \( T \) if
\[\liminf_{N \to \infty} \frac{1}{N} \sum_{j=1}^N \| T^j x \| = 0
\quad \text{and} \quad \limsup_{N \to \infty} \frac{1}{N} \sum_{j=1}^N \| T^j x \| = \infty\ \  \mbox{(resp. $>0$)}.\]
\end{definition}

\subsection{Power-bounded and absolutely  Ces\`aro bounded operators }
We now introduce the two boundedness conditions central to this paper. As we shall see, power-boundedness and absolute Ces\`aro boundedness are not merely technical hypotheses: the failure of the former is equivalent to Li--Yorke chaos, and the failure of the latter
is equivalent to mean Li--Yorke chaos.

\begin{definition}
An operator  \( T \in \mathcal{L}(X) \) is said to be \textit{power-bounded} if there exists \( C \in (0, \infty) \) such that \( \|T^n\| \leq C \) for all \( n \in \mathbb{N} \). More generally, given a subspace \( Z \) of \( X \), we say that \( T \) is \textit{power-bounded in \( Z \)} if there exists \( C \in (0, \infty) \) such that
\[
\|T^n|_Z\| \leq C \quad \text{for all } n \in \mathbb{N}.
\]
\end{definition}

\begin{definition}
An operator \( T \) is said to be \emph{absolutely Ces\`aro bounded} if there exists a constant \( C > 0 \) such that
\[
\sup_{N \in \mathbb{N}} \frac{1}{N} \sum_{j=1}^N \| T^j x \| \leq C \| x \| \quad \text{for all } x \in X.
\]
\end{definition}

\begin{theorem}[See {\cite[Theorem~4]{BernardesPerisSanchez}}]\label{CBthm}
For every \( T \in \mathcal{L}(X) \), the following assertions are equivalent:
\begin{itemize}
    \item[(i)] \( T \) is not absolutely Ces\`aro bounded;
    \item[(ii)] There is a vector \( x \in X \) such that
    $\displaystyle\sup_{N \in \mathbb{N}} \frac{1}{N} \sum_{j=1}^N \| T^j x \| = \infty$;
    \item[(iii)] The set of all vectors \( y \in X \) such that $\displaystyle\sup_{N \in \mathbb{N}} \frac{1}{N} \sum_{j=1}^N \| T^j y \| = \infty$
    is residual in \( X \).
\end{itemize}
\end{theorem}

\begin{theorem}[See {\cite[Theorem~22]{BernardesPerisSanchez}}]\label{MLYthm}
If \( T \in \mathcal{L}(X) \) and 
\begin{align}\label{setACB}
Y := \left\{ x \in X : \liminf_{N \to \infty} \frac{1}{N} \sum_{j=1}^N \|T^j x\| = 0 \right\}    
\end{align}
is dense in \( X \), then the following assertions are equivalent:
\begin{itemize}
    \item [i)] \( T \) is mean Li--Yorke chaotic;
    \item [ii)] \( T \) has a residual set of absolutely mean irregular vectors;
    \item [iii)] \( T \) is not absolutely Ces\`aro
 bounded;
   \item [iv)] $\displaystyle \limsup_{N \to \infty} \frac{1}{N} \sum_{j=1}^{N} \|T^j y_0\| > 0 \quad \text{for some } y_0 \in X.$
\end{itemize}
\end{theorem}
In the case where \( X \) is separable, condition \( i) \) may be replaced with the assumption that the operator is densely mean Li--Yorke chaotic.

\subsection{Li--Yorke chaotic operators are not power-bounded }

From now on \( X \) will denote an infinite-dimensional separable Banach space. Let \( T \in \mathcal{L}(X) \) and define
\begin{align*}\label{setliminf}
\mathrm{NS}(T):=\left\{ x \in X : (T^n x)_{n \in \mathbb{N}} \text{ has a subsequence converging to zero} \right\}    
\end{align*}
 and \( \mathcal{X} \) the closure of the linear span of \( \mathrm{NS}(T) \).  It is clear that \( \mathcal{X} \) is non-empty and invariant under \( T \); that is, \( T|\mathcal{X} \in \mathcal{L}(\mathcal{X}) \).

\begin{proposition}\label{GenLY}
Let \( T \in \mathcal{L}(X) \) be a bounded operator. Then \( T \) is Li-Yorke chaotic if and only if it is not power bounded in \(\mathcal{X}\). 
\end{proposition}
\begin{proof}
Suppose that \(T \) is Li-Yorke chaotic. By the LYCC, there exists an increasing sequence of integers \( (n_k)_k \), \( X_0 \subset X \) such that:
$\displaystyle\lim_{k \to \infty} T^{n_k} x = 0,$ for all $x \in X_0,$ and \( T^n|_Y \) not power bounded, where \( Y := \operatorname{span}(X_0) \). Since $Y\subset\mathcal{X}$, it follows that $T^n |_\mathcal{X}$ is not power bounded as well.

Conversely, suppose that \( T \) is not Li-Yorke chaotic. Let $x\in X$ whose orbit under \( T \) has a subsequence converging to zero; this implies that $\displaystyle\lim_{n \to \infty} \inf \| T^n x \| = 0$. Since $T$ has no irregular vectors, it follows that $\displaystyle\limsup_{n\to\infty} \|T^n x\| = 0$ as well, and therefore $\displaystyle\lim_{n\to\infty} T^n x = 0$.  
In this setting, the set \( \mathrm{NS}(T) \) defines a subspace. If \( \mathcal{X} \) is finite-dimensional, \( \mathcal{X}= \mathrm{NS}(T)\), thus \( T \) is power-bounded  in \( \mathcal{X} \). Otherwise, \( \mathrm{NS}(T) \) is dense and by Proposition~\ref{resconv}, we conclude that  \( \mathrm{NS}(T) \) is residual in \( \mathcal{X}\). Now define \( \mathcal{R} \) as the set of all vectors \( x \in \mathcal{X} \) whose orbit under \( T \) is unbounded. If \( T \) was not power-bounded  in \( \mathcal{X} \), by the 
Banach--Steinhaus \cite[Theorem~2.5]{RudinFunctional}, \( \mathcal{X} \setminus \mathcal{R} \) is a proper subset of first category on $X$, hence \( \mathcal{R} \) is residual in \( \mathcal{X} \). Since the intersection  \( \mathrm{NS}(T) \cap \mathcal{R} \) is residual in \( \mathcal{X} \) (and thus dense in \( \mathcal{X} \)), any element \( y \in  \mathrm{NS}(T) \cap \mathcal{R} \) would be an irregular vector. Since we are assuming that $T$ is not
Li-Yorke chaotic, we conclude that is power bounded \( \mathcal{X} \).
\end{proof}

\begin{lemma}\label{eqLYH2}
Let \( T \in \mathcal{L}(X) \) be an operator. If the set \( \mathrm{NS}(T) \)
is dense in $X$, then the following assertions are equivalent:
\begin{itemize}
    \item[(i)] \( T \) is densely Li-Yorke chaotic;
    \item[(ii)] \( T \) has a residual set of irregular vectors;
    \item[(iii)] \( T \) is not power-bounded.
\end{itemize}
\end{lemma}
\begin{proof}
If \( \text{(iii)}\) holds, proceeding as the proof of Proposition~\ref{GenLY}, one can show that the set \( \mathcal{R}_2 \) of all vectors \( x \in X \) whose orbit under \( T \) is unbounded is residual. Moreover, by the hypothesis, the set \( \mathcal{R}_1 \) of all vectors \( x \in X \) whose orbit under \( T \) has a subsequence converging to zero is dense in \( X \), and therefore residual by Proposition~\ref{resconv}; then \( \mathcal{R}_1 \cap \mathcal{R}_2 \) is residual and every vector is irregular, which implies \(\text{(ii)} \). \( \text{(ii)} \Rightarrow \text{(i)}\) follows by {\cite[Theorem 10]{Bernardes2015}}, while \( \text{(i)} \Rightarrow \text{(iii)} \) follows once again from Proposition~\ref{GenLY}.
\end{proof}

\section{Li--Yorke chaos for weighted composition operators on spaces of analytic functions}\label{Sec3}

\subsection{Setup and structural conditions}

Let \( H^\infty \) denote the Banach space of bounded analytic functions on \( \mathbb{D} \), equipped with the supremum norm
\[
\|f\|_{\infty} = \sup\{ |f(z)| : z \in \mathbb{D} \}.
\]

Given an analytic self-map \( \varphi \colon \mathbb{D} \to \mathbb{D} \) and $w \in H^\infty$, we associate to \( w \) and \( \varphi \) the following sequence of continuous maps from \( \mathbb{D}\) into \( \mathbb{C} \):
\[w^{(1)} := w \quad \text{and} \quad w^{(n)} := (w \circ \varphi^{n-1}) \cdot \ldots \cdot (w \circ \varphi) \cdot w, \quad \text{for } n \geq 2.\]
By simple definition, it is possible to prove that 
\[
(C_{w,\varphi})^n(f) =  w^{(n)} \cdot (f \circ\varphi^n). 
\]
We observe that if \( w \in H^\infty \), then \( w^{(n)} \) also belong to \( H^\infty \) for all \( n \geq 1 \). Indeed,
\begin{align*}
|w^{(n)}(z)| = |(w \circ \varphi^{n-1})(z) \cdot \ldots \cdot (w \circ \varphi)(z) \cdot w(z)| \leq \|w\|_{\infty}^n, \quad \forall z \in \mathbb{D}.
\end{align*}

The goal is to establish necessary and sufficient conditions for $C_{w,\varphi}$ to be Li--Yorke (resp. mean Li--Yorke) chaotic. To this end, let $X$ be an infinite-dimensional separable Banach space of analytic functions on $\mathbb{D}$ satisfying:

\begin{mdframed}[linewidth=1pt, roundcorner=5pt]
\begin{enumerate}[label=C.\arabic*]
\item $H^\infty$ is dense in $X$.\label{C1}
\item Every weighted composition operator $C_{w,\varphi}$ on $X$, where \( \varphi : \mathbb{D} \to \mathbb{D} \) is an analytic self-map and $w \in H^\infty$, is bounded and satisfies 
    \[
        \|(C_{w,\varphi})^n f\|_X \leq \|w^{(n)}\|_X \|f\|_\infty, \quad \forall f \in H^\infty,\quad \forall n \in \mathbb{N}.
    \]\label{C2}
\end{enumerate}
\end{mdframed}

From now on $X$ will denote a separable Banach space satisfying (\ref{C1}) and (\ref{C2}), \( \varphi : \mathbb{D} \to \mathbb{D} \) an analytic self-map, and $w \in H^\infty$.

The following results provide the justification for the conditions given above.

\begin{theorem}\label{sufLY}
Let $C_{w,\varphi}$ be a weighted composition operator on $X$. If
\[
    \liminf_{n \to \infty} \| w^{(n)} \|_X = 0,
\]
then the following assertions are equivalent:
\begin{itemize}
    \item[(i)] $C_{w,\varphi}$ is densely Li--Yorke chaotic.
    \item[(ii)] $C_{w,\varphi}$ has a residual set of irregular vectors.
    \item[(iii)] $C_{w,\varphi}$ is not power bounded.
\end{itemize}    
\end{theorem}
\begin{proof}
By (\ref{C2}) and the hypothesis, $\liminf_{n\to\infty}\|(C_{w,\varphi})^n f\|_X = 0$ for all $f \in H^\infty$. Thus the set $\mathrm{NS}(C_{w,\varphi})$ contains $H^\infty$, and is dense in $X$ by (\ref{C1}). By Lemma \ref{eqLYH2}, the conclusion follows.
\end{proof}

As a consequence of the previous theorem, we obtain a sufficient condition for \( C_{w,\varphi} \) to be Li--Yorke chaotic, expressed entirely in terms of its symbols.

\begin{corollary}\label{sufHP}
Let $C_{w,\varphi}$ be a weighted composition operator on $X$. Then \( C_{w,\varphi} \) is densely Li-Yorke chaotic if  there is an increasing sequence $(n_k)_{k\in\mathbb{N}}$ of positive integers such that
\begin{itemize}
    \item[i)] $\lim_{k \to \infty}\| w^{(n_k)} \|_X = 0$, and
    \item[ii)] $\sup \{ \| w^{(n)} \|_X: n\in \mathbb{N}\}=\infty.$
\end{itemize}
\begin{proof}
Note that condition \( i) \) implies  
$\displaystyle\lim_{n \to \infty} \inf \| w^{(n)} \|_X = 0.$
Moreover, observe that $\| w^{(n)} \|_X = \| (C_{w,\varphi})^n 1 \|_X \leq \left\| (C_{w,\varphi})^n  \right\|_X.$
Therefore, condition \( ii) \) implies that \( C_{w,\varphi} \) is not power bounded in \( X \). Hence, by Theorem~\ref{sufLY}, the conclusion follows.
\end{proof}
\end{corollary}

The next result provides necessary and sufficient conditions for determining whether a weighted composition operator \( C_{w,\varphi} \) on \( X \) is mean Li--Yorke chaotic.

\begin{theorem}\label{sufMLY}
Let $C_{w,\varphi}$ be a weighted composition operator on $X$. If
\[
\liminf_{N \to \infty} \dfrac{1}{N} \sum_{j=1}^{N}\| w^{(j)} \|_X = 0,
\]
then the following assertion are equivalent:
\begin{itemize}
    \item [i)] \( C_{w,\varphi} \) is densely mean Li--Yorke chaotic;
    \item [ii)] \( C_{w,\varphi} \) has a residual set of absolutely mean irregular vectors;
    \item [iii)] \( C_{w,\varphi} \) is not absolutely Ces\`aro bounded.
    \item [iv)] $\displaystyle \limsup_{N \to \infty} \frac{1}{N} \sum_{j=1}^{N} \|(C_{w,\varphi})^j y_0\|_X > 0 \quad \text{for some } y_0 \in X.$
\end{itemize}
\end{theorem}
\begin{proof}
By (\ref{C2}) and the hypothesis, $\displaystyle\liminf_{N\to\infty}\frac{1}{N}\sum_{j=1}^N\|(C_{w,\varphi})^j f\|_X = 0$ for all $f \in H^\infty$. Hence the set
\[
Y = \left\{ f \in X : \liminf_{N \to \infty} \frac{1}{N} \sum_{j=1}^N \|(C_{w,\varphi})^j f\|_X = 0 \right\}
\]
contains $H^\infty$, and is dense in $X$ by (\ref{C1}). Since $X$ is separable, the conclusion follows from Theorem~\ref{MLYthm}.

\end{proof}

As a consequence of the previous theorem, we obtain a sufficient condition for \( C_{w,\varphi} \) to be mean Li--Yorke chaotic, expressed entirely in terms of its symbols.

\begin{corollary}\label{sufMLYHP}
Let $C_{w,\varphi}$ be a weighted composition operator on $X$. Then \( C_{w,\varphi} \) is densely mean Li-Yorke chaotic if  there is an increasing sequence $(N_k)_{k\in\mathbb{N}}$ of positive integers such that
\[
\displaystyle\lim_{k \to \infty}\dfrac{1}{N_k}\sum_{j=1}^{N_k} \| w^{(j)} \|_X = 0\quad\text{and}\quad\displaystyle\sup_{N\in \mathbb{N}}  \dfrac{1}{N}\sum_{j=1}^{N}\| w^{(j)} \|_X=\infty.
\]
\end{corollary}
\begin{proof}
Note that condition \( i) \) implies  $\displaystyle\lim_{N \to \infty} \inf \dfrac{1}{N}\displaystyle\sum_{j=1}^{N}\| w^{(j)} \|_X = 0.$ Moreover, observe that $\dfrac{1}{N}\displaystyle\sum_{j=1}^{N}\| w^{(j)} \|_X = \dfrac{1}{N}\sum_{j=1}^{N}\| (C_{w,\varphi})^j 1 \|_X.$
Therefore, condition \( ii) \) and Theorem \ref{CBthm} implies that \( C_{w,\varphi} \) is not absolutely Cesàro bounded. Hence, by Theorem~\ref{sufMLY}, the conclusion follows.
\end{proof}

In the following subsections, we investigate which spaces $X$ satisfy conditions (\ref{C1}) and (\ref{C2}).

\subsection{The Hardy spaces \texorpdfstring{$H^p$}{Hp} for \texorpdfstring{$1\leq p<\infty$}{1<p<infinity}}

Let \( \mathbb{T} \) denote the unit circle in the complex plane. The Hardy space \( H^p \) (\( p > 0 \)) is the space consisting of functions \( f \) holomorphic in \( \mathbb{D} \) for which
\[
\|f\|_{p}^p := \sup_{0 < r < 1} \int_{\mathbb{T}} |f(rz)|^p \, dm(z) < \infty
\]
where \( m \) is normalized Lebesgue measure on \( \mathbb{T} \). The \( H^p \) spaces are separable Banach spaces for \( 1\leq p <\infty \).

The space $H^2$ is a Hilbert space with inner product
\[
\langle f,g \rangle = \sum_{n=0}^{\infty} a_n \overline{b_n},
\]
where $(a_n)_{n\in\mathbb{N}}$ and $(b_n)_{n\in\mathbb{N}}$ are the Maclaurin coefficients for $f$ and $g$ respectively.

If the weight function \( w \) belongs to \( H^\infty \) and $\varphi$ is an analytic self-map on $\mathbb{D}$, it is known that $\varphi$ induces a bounded composition operator on $H^p$ (see \cite{Ryff1966}) and the multiplication operator induced by $w$ is also bounded on $H^p$, then \( C_{w,\varphi} \) is a bounded operator on \( H^p \).

The next Lemma shows that condition (\ref{C2}) holds for $X=H^p$.

\begin{lemma}\label{lemwn}
For \( 1\leq p <\infty \), let \( C_{w,\varphi}\) be a weighted composition operator on \( H^p \). Then
\[
\| (C_{w,\varphi})^n f \|_p \leq \| w^{(n)} \|_p\| f\|_\infty,  \quad \forall f\in H^\infty,  \quad \forall n \in \mathbb{N}.
\]
\end{lemma}
\begin{proof}
Note first that
\begin{align*}
\|  w^{(n)} \cdot (f \circ\varphi^n)\|^p_p &=\sup_{0 < r < 1} \int_{\mathbb{T}} |w^{(n)}(rz)(f \circ\varphi^n)(rz)|^p \, dm(z)\\ & \leq  \| f\|_{\infty}^p\sup_{0<r<1}\int_{\mathbb{T}}| w^{(n)}(rz)|^p \ dm(z) \leq  \| w^{(n)} \|^p_p \|f\|_\infty^p.  
\end{align*}
Therefore, $\| (C_{w,\varphi})^n f\|_p  \leq \| w^{(n)} \|_p \|f\|_\infty, \ \ \forall f\in H^\infty.   $ 
\end{proof}

It is worth noting that the closure of \( H^\infty \) in the \( H^p \)-norm topology is the whole space \( H^p \). Indeed, since $\operatorname{span} \{ z^n \}_{n \geq 0} \subset H^\infty \quad \text{and} \quad H^p = \overline{\operatorname{span} \{ z^n \}_{n \geq 0}}^{H^p},$ (see \cite{Duren1970, Rosemblum1994} for instance) it follows that  \( \overline{H^\infty}^{H^p} = H^p \). Therefore, condition (\ref{C1}) holds.

\subsection{The weighted Bergman spaces \texorpdfstring{$A^p_\beta$}{A2beta} for \texorpdfstring{$-1<\beta<\infty$ and $1<p<\infty$}{-1<beta<infinity} }

Let \( dA(z) \) be the normalized area measure on \( \mathbb{D} \), and \( -1 < \beta < \infty \). The weighted Bergman space \( A^p_\beta(\mathbb{D}) \) is the space of all analytic functions in \( \mathbb{D} \) such that the norm
\[
\|f\|^p_{\beta_p} :=  \int_{\mathbb{D}} |f(z)|^p \, dA_\beta(z) < \infty,
\]
where \( dA_\beta(z) = (\beta + 1)(1 - |z|^2)^\beta \, dA(z) \). The weighted Bergman space \( A^p_\beta \) is a separable infinite-dimensional Banach space. It is straightforward to verify that \( H^\infty \) is a subspace of \( A^p_\beta \) and \( \|f\|_{\beta_p} \leq C \|f\|_\infty \) for some $C>0$. 

The weighted Bergman space $A^2_\beta$ is a Hilbert space with the inner product
\begin{equation*}
\langle f,g \rangle
=
\sum_{n=0}^{\infty}
\frac{n!\,\Gamma(2+\beta)}{\Gamma(n+2+\beta)}
\, \widehat{f}(n)\,\overline{\widehat{g}(n)},
\end{equation*}
where $(\widehat{f}(n))_{n\in\mathbb{N}}$ and $(\widehat{g}(n))_{n\in\mathbb{N}}$ are the sequences of Maclaurin coefficients for $f$ and $g$ respectively, and $\Gamma$ is the Gamma function.

If the weight function \( w \) belongs to \( H^\infty \) and $\varphi$ is an analytic self-map on $\mathbb{D}$, it is known that $\varphi$ induces a bounded composition operator on \( A^p_\beta\) (see \cite[Proposition~7.1]{Blasco1992}) and the multiplication operator induced by $w$ is also bounded on \( A^p_\beta \), then \( C_{w,\varphi} \) is a bounded operator on \( A^p_\beta \).

The next Lemma shows that condition (\ref{C2}) holds for $X=A^p_\beta$.

\begin{lemma}\label{lemwnWBS}
For \( 1< p <\infty \), let \( C_{w,\varphi}\) be a weighted composition operator on \( A^p_\beta \). Then
\[
\| (C_{w,\varphi})^n f \|_{\beta_p} \leq \| w^{(n)} \|_{\beta_p}\| f\|_\infty,  \quad \forall f\in H^\infty,  \quad \forall n \in \mathbb{N}.
\]
\end{lemma}
\begin{proof}
Note first that
\begin{align*}
\|  w^{(n)} \cdot (f \circ\varphi^n)\|^p_{\beta_p} &=\int_{\mathbb{D}} |w^{(n)}(z)(f \circ\varphi^n)(z)|^p \, dA_\beta(z) & \leq  \| f\|_{\infty}^p\int_{\mathbb{D}}| w^{(n)}(z)|^p \ dA_\beta(z)\\
& \leq  \| w^{(n)} \|^p_{\beta_p} \|f\|_\infty^p.  
\end{align*}
Therefore, $\| (C_{w,\varphi})^n f\|_{\beta_p}  \leq \| w^{(n)} \|_{\beta_p} \|f\|_\infty, \ \ \forall f\in H^\infty.$ 
\end{proof}

It is worth noting that the closure of \( H^\infty \) in the \( A^p_\beta \)-norm topology is the whole space \( A^p_\beta \). Indeed, since $\operatorname{span} \{ z^n \}_{n \geq 0} \subset H^\infty$ and the polinomials are dense in  $A^p_\beta$ it follows that \( \overline{H^\infty}^{A^p_\beta} = A^p_\beta \) (see \cite{HKZ2000} for a reference).  Therefore, condition (\ref{C1}) holds.


\begin{corollary}
Suppose that there is an increasing sequence $(n_k)_{k\in\mathbb{N}}$ of positive integers such that
\begin{itemize}
    \item[i)] $\lim_{k \to \infty}\| w^{(n_k)} \|_2 = 0$, and
    \item[ii)] $\sup \{ \| w^{(n)} \|_{\beta_2}: n\in \mathbb{N}\}=\infty$
\end{itemize}
Then the weighted composition operator \( C_{w,\varphi} \) is densely Li--Yorke chaotic on both \( H^2 \) and \( A^2_{\beta} \).
\begin{proof}
According to \cite{Gori2006}, we have $\|f\|_{\beta} \leq C \|f\|_{2}$, for all $f \in H^{2}$. Then the result follows by Corollary \ref{sufHP} for \(X = A^{2}_{\beta}\) and \(X = H^{2}\), respectively.
\end{proof}
\end{corollary}

\begin{corollary}
Suppose that there is an increasing sequence $(N_k)_{k\in\mathbb{N}}$ of positive integers such that
\[
\displaystyle\lim_{k \to \infty}\dfrac{1}{N_k}\sum_{j=1}^{N_k} \| w^{(j)} \|_2 = 0\quad\text{and}\quad\displaystyle\sup_{N\in \mathbb{N}}  \dfrac{1}{N}\sum_{j=1}^{N}\| w^{(j)} \|_{\beta_2}=\infty.
\]
Then the weighted composition operator \( C_{w,\varphi} \) is densely mean Li--Yorke chaotic on both \( H^2 \) and \( A^2_\beta \).
\end{corollary}
\begin{proof}
According to \cite{Gori2006}, we have $\|f\|_{\beta} \leq C \|f\|_{2}$, for all $f \in H^{2}$. Then the result follows by Corollary \ref{sufMLYHP} for \(X = A^{2}_{\beta}\) and \(X = H^{2}\), respectively.
\end{proof}

\section{The \texorpdfstring{$H^\infty$}{Hinfinity} space}\label{Sec4}

We now turn our attention to weighted composition operators on \( H^\infty \). Although $H^\infty$ satisfies conditions (\ref{C1}) and (\ref{C2}), since \( H^\infty \) is not separable the results of Section~\ref{Sec3} do not apply directly. Nevertheless, the underlying ideas can be suitably adapted to this setting.

\begin{lemma}\label{lemCinf}
For any weighted composition operator \( C_{w,\varphi} \) on \( H^\infty \):
 \[
 \| (C_{w,\varphi})^n  \|= \| w^{(n)} \|_{\infty}, \quad \forall n \in \mathbb{N}.
 \]
\end{lemma}
\begin{proof}
It is clear by the definition of the norm that
 \[
\| (C_{w,\varphi})^n f \|_\infty \leq \| w^{(n)} \|_{\infty}\| f\|_\infty, \quad \forall n \in \mathbb{N}.
 \]   
 Then $ \| (C_{w,\varphi})^n  \|\leq \| w^{(n)} \|_{\infty}.$ On the other hand,
 \[
\| w^{(n)} \|_{\infty} = \| (C_{w,\varphi})^n 1 \|_{\infty}\leq  \| (C_{w,\varphi})^n  \|, \quad \forall n \in \mathbb{N}.
 \]    
\end{proof}

\begin{theorem}\label{HinfWCO}
Consider a weighted composition operator \( C_{w,\varphi} \) on \( H^\infty  \). If $$\displaystyle\liminf_{n \to \infty} \| w^{(n)} \|_\infty = 0,$$ then the following assertions are equivalent:
\begin{itemize}
    \item[(i)] \( C_{w,\varphi} \) is Li-Yorke chaotic;
    \item[(ii)] \( C_{w,\varphi} \) has a residual set of irregular vectors;
    \item[(iii)] \( C_{w,\varphi} \) is not power-bounded.
\end{itemize}
\end{theorem}
\begin{proof}
By Lemma~\ref{lemCinf} and hypothesis we have that
\[
\liminf_{n \to \infty} \| (C_{w,\varphi})^n f \|_{\infty} = 0, \quad \forall f \in H^{\infty}.
\]
If \( \text{(iii)}\) holds, proceeding as the proof of Proposition~\ref{GenLY}, one can show that the set \( \mathcal{R} \) of all vectors \( f \in H^\infty \) whose orbit under \( C_{w,\varphi} \) is unbounded is residual. Moreover, by hypothesis, every vector in \(\mathcal{R}\) is irregular, which implies \(\text{(ii)}\). The implication \(\text{(ii)} \Rightarrow \text{(i)}\) is straightforward, and \(\text{(i)} \Rightarrow \text{(iii)}\) follows as in Proposition~\ref{GenLY}.
\end{proof}

As a consequence of the previous theorem, we obtain a sufficient condition for \( C_{w,\varphi} \) to be Li--Yorke chaotic, expressed entirely in terms of its symbols.

\begin{corollary}
Let \( C_{w,\varphi} \) be a weighted composition operator on \( H^\infty \). Then \( C_{w,\varphi} \) is Li-Yorke chaotic if  there is an increasing sequence $(n_k)_{k\in\mathbb{N}}$ of positive integers such that
\begin{itemize}
    \item[i)] $\displaystyle\lim_{k \to \infty}\| w^{(n_k)} \|_\infty = 0$, and
    \item[ii)] $\displaystyle\sup \{ \| w^{(n)} \|_\infty: n\in \mathbb{N}\}=\infty.$
\end{itemize}
\end{corollary}

The next result provides necessary and sufficient conditions for determining whether a weighted composition operator \( C_{w,\varphi} \) on \( H^\infty \) is mean Li--Yorke chaotic.

\begin{theorem}\label{MLYHinf}
Consider a weighted composition operator \( C_{w,\varphi} \) on \( H^\infty  \). If
\[
\liminf_{N \to \infty} \dfrac{1}{N}\sum_{j=1}^{N}\| w^{(j)} \|_\infty = 0,
\]
then the following assertion are equivalent:
\begin{itemize}
    \item [i)] \( C_{w,\varphi} \) is mean Li--Yorke chaotic;
    \item [ii)] \( C_{w,\varphi} \) has a residual set of absolutely mean irregular vectors;
    \item [iii)] \( C_{w,\varphi} \) is not absolutely Ces\`aro bounded.
    \item [iv)] $\displaystyle \limsup_{N \to \infty} \frac{1}{N} \sum_{j=1}^{N} \|(C_{w,\varphi})^j y_0\|_{\infty} > 0 \quad \text{for some } y_0 \in H^\infty.$
\end{itemize}

\end{theorem}
\begin{proof}
By Lemma~\ref{lemCinf} and the hypothesis, $\displaystyle\liminf_{N\to\infty}\frac{1}{N}\sum_{j=1}^N\|(C_{w,\varphi})^j f\|_\infty = 0$ for all $f \in H^\infty$. The conclusion follows from Theorem~\ref{MLYthm}.
\end{proof}

\begin{corollary}
Let \( C_{w,\varphi} \) be a weighted composition operator on \( H^\infty \). Then \( C_{w,\varphi} \) is mean Li-Yorke chaotic if  there is an increasing sequence $(N_k)_{k\in\mathbb{N}}$ of positive integers such that 
\[
\displaystyle\lim_{k \to \infty}\dfrac{1}{N_k}\sum_{j=1}^{N_k} \| w^{(j)} \|_\infty = 0 \quad \text{and}\quad \displaystyle\sup_{N\in \mathbb{N}}  \dfrac{1}{N}\sum_{j=1}^{N}\| w^{(j)} \|_\infty=\infty.
\]

\end{corollary}

\section{Li-Yorke chaotic weighted compostion operators on \texorpdfstring{$H^p$}{}  and \texorpdfstring{$A^2_\beta$}{} }\label{Sec5}
This section is devoted to examples. We present an example illustrating the results established in Section \ref{Sec3} for weighted composition operators on Hardy and Bergman spaces. On the other hand, we provide an example where the condition in Theorem~\ref{sufLY} is not strictly necessary to obtain Li-Yorke chaos.\vspace{0.3cm}

Consider the weight function $w_\lambda(z) = \lambda z$. A straightforward computation shows that
$$
\|w_\lambda^{(n)}\|_\infty \leq |\lambda|^n,
$$
so in particular $\|w_\lambda^{(n)}\|_\infty \to 0$ whenever $|\lambda| < 1$. Next, consider the family of affine self-maps of the unit disk given by
$$
\varphi_a(z) = az + 1 - a, \quad 0 < a < 1.
$$
\begin{proposition}
Let $0 < a < 1$ and suppose that $a^{1/2} < |\lambda| < 1$. Then the weighted composition operator $C_{w_\lambda,\varphi_a}$ is densely Li--Yorke chaotic on $A^2_\beta$ and $H^p$ for $1\leq p\leq2$.
\end{proposition}
\begin{proof}
First, choose $\mu\in \mathbb{R}$ such that $1 < \mu < a^{-1/2}$ and $|\mu\cdot\lambda| > 1$. It was shown in \cite{CarmoNoor2022} that the point spectrum of the composition operator $C_{\varphi_a}$ on the Hardy space $H^2$ contains the annulus $\{\mu \in \mathbb{C} : 0 < |\mu| < a^{-1/2}\}$. Then there exists $f_\mu \in H^2$ satisfying $C_{\varphi_a}f_\mu = \mu f_\mu$. Since $C_{\varphi_{a^n}}w_\lambda - \lambda = \lambda a^n(z-1)$, the sequence $C_{\varphi_{a^n}}w_\lambda$ converges to the constant function $\lambda$ in the $H^\infty$-norm. Because $|\mu\cdot\lambda|>1$, we can find $n_0 \in \mathbb{N}$ and $\gamma < |\lambda|$ with $| \mu\cdot \gamma| > 1$, such that $|C_{\varphi_{a^n}}w_\lambda(z)| > \gamma$ for all $z \in \mathbb{D}$ and all $n > n_0$.

For $n > n_0$, this yields the pointwise estimate
$$
|(C_{w_\lambda,\varphi_a})^n f_\mu(z)| \geq |\mu\cdot\gamma|^{n-n_0}\,|(C_{w_\lambda,\varphi_a})^{n_0}f_\mu(z)|,
$$
and since $(C_{w_\lambda,\varphi_a})^{n_0}f$ is not identically zero, passing to the $H^2$-norm gives
$$
\|(C_{w_\lambda,\varphi_a})^n f_\mu\|_2 \geq |\mu\cdot\gamma|^{n-n_0}\,\|(C_{w_\lambda,\varphi_a})^{n_0}f_\mu\|_2.
$$
As $|\mu\cdot\gamma|>1$, the right-hand side grows without bound, showing that $C_{w_\lambda,\varphi_a}$ is not power bounded. Theorem~\ref{sufLY} then implies that $C_{w_\lambda,\varphi_a}$ is densely Li--Yorke chaotic on $H^2$. Finally, since $H^2 \subseteq  A^2_{\beta}$ and $H^2 \subseteq H^p$ for $1\leq p\leq 2$, the same argument carries over to show that $C_{w_\lambda,\varphi_a}$ is densely Li--Yorke chaotic on $ A^2_{\beta}$ and $H^p$ for all $1\leq p\leq 2$.

\end{proof}

In what follows, we show that the condition $\|w^{(n)}\|_X \to 0$ in Theorem~\ref{sufLY} is not necessary for a weighted composition operator to be Li--Yorke chaotic.

\begin{proposition}
Let $0 < a < 1$. Then the composition operator $C_{\varphi_a}$ is Li--Yorke chaotic on $H^2$ and $ A^2_{\beta}$.    
\end{proposition}
\begin{proof}
In \cite{CarmoNoor2022} it was shown that the functions $f_s(z)=(z-1)^s$ are eigenvectors of $C_{\varphi_a}$ and satisfy
\[
C_{\varphi_a}f_s=a^s f_s,
\]
where $f_s\in H^2$ if and only if $\operatorname{Re}(s)>-\frac{1}{2}$. It is straightforward to verify that 
\begin{itemize}
\item $\|C_{\varphi_a}^n f_s\|_2 \to 0$ as $n \to \infty$ if $\mathrm{Re}(s) > 0$.

\item $\|C_{\varphi_a}^n f_s\|_2 \to \infty$ as $n \to \infty$ if $\mathrm{Re}(s) < 0$.
\end{itemize}
Let us check that, for each $k\in\mathbb{N}$, $f_{1/4}z^k\in \text{NS}(C_{\varphi_a})$. Indeed:
\begin{align*}
    \| C_{\varphi_a}^nf_{1/4}z^k\|_2 &=\| C_{\varphi_a}^nf_{1/4}C_{\varphi_a}^nz^k\|_2=\| a^{n/4}f_{1/4}C_{\varphi_a}^nz^k\|_2\\ &\leq a^{n/4}\|f_{1/4}\|_\infty\| C_{\varphi_{a^n}}z^k\|_\infty \\ &= a^{n/4}\|f_{1/4}\|_\infty\| (a^nz+1-a^n)^k\|_\infty\\ &\leq a^{n/4}\|f_{1/4}\|_\infty (a^n+1)^k.
\end{align*}
Therefore, $\| C_{\varphi_a}^nf_{1/4}z^k\|_2\rightarrow0$ as $n\rightarrow\infty$. Since $f_{1/4}\in H^\infty$, it follows that $$f_{1/4} H^2\subseteq \mathcal{X}=\overline{\operatorname{span}}(\mathrm{NS}(C_{\varphi_a})).$$ In particular, $f_{-1/12}=f_{1/4} f_{-1/3}\in \mathcal{X}$,
and since 
\[
\left\| C_{\varphi_a}^{\,n} f_{-1/12} \right\|_2 \to \infty
\quad \text{as } n \to \infty,
\]
we conclude that $C_{\varphi_a}$ is Li--Yorke chaotic by Proposition \ref{GenLY}. The same argument holds for $A^2_{\beta}$.
\end{proof}

Recall that $\mathbb{C}_{+}:=\{z\in \mathbb{C}: \text{Re}(z)>0\}$ represents the open right half-plane. In \cite[Theorem~10]{AlvarezHenriquez2025} it was proved that no composition operator with linear fractional symbol on the Hardy space $H^{2}(\mathbb{C}_{+})$ can be Li-Yorke chaotic. We now exhibit a family of Li--Yorke chaotic operators which are a constant multiple of composition operators on $H^2(\mathbb{C}_+)$.

\begin{corollary}
Let $0 < a < 1$. Then the operator $a^{-1}C_{\psi_a}$ is Li--Yorke chaotic on $H^2(\mathbb{C}_+)$, where
\[
\psi_a(w) = a^{-1}w + (a^{-1} - 1).
\].      
\end{corollary}
\begin{proof}
According to \cite{CarmoNoor2022}, \( C_{\varphi_a} \) on \( H^2\) is unitarily equivalent to $a^{-1} C_{\psi_a}$ on $H^2(\mathbb{C}_+)$.
\end{proof}

\bibliography{Bibliography}
\bibliographystyle{acm}

\end{document}